\documentclass{nyjm}
\usepackage{amssymb,latexsym,amsmath,amsthm,amsfonts}
\usepackage{booktabs}
\usepackage{hyperref}
\hypersetup{nesting=true,debug=true,naturalnames=true}
\usepackage{graphicx,amssymb,upref}

\theoremstyle{plain}
\newtheorem{thm}{Theorem}
\newtheorem{cor}[thm]{Corollary}
\newtheorem{lem}[thm]{Lemma}
\newtheorem{prop}[thm]{Proposition}
\newtheorem{conj}[thm]{Conjecture}

\theoremstyle{definition}

\newtheorem{ex}[thm]{Example}
\newtheorem{rem}[thm]{Remark}

\numberwithin{equation}{section}
\numberwithin{thm}{section}
\numberwithin{table}{section}

\newcommand{\abs}[1]{\left\vert#1\right\vert}
\newcommand{\set}[1]{\left\{#1\right\}}
\newcommand{\gen}[1]{\left\langle#1\right\rangle}
\DeclareMathOperator{\charpol}{charpol}
\DeclareMathOperator{\minpol}{minpol}

\DeclareMathOperator{\GL}{GL}

\DeclareMathOperator{\GSp}{GSp}

\DeclareMathOperator{\Gal}{Gal}
\DeclareMathOperator{\Frob}{Frob}
\DeclareMathOperator{\disc}{disc}

\DeclareMathOperator{\cond}{cond}

\newcommand{\A}{\mathcal A}

\newcommand{\Z}{\mathbb Z}
\newcommand{\Q}{\mathbb Q}

\newcommand{\F}{\mathbb F}
\newcommand{\I}{\rm I}
\newcommand{\Fp}{\mathbb F_p}
\newcommand{\Fq}{\mathbb F_q}
\newcommand{\Fell}{\mathbb F_\ell}
\newcommand{\cclass}{\mathcal{C}}
\newcommand{\cent}{\mathcal Z}

\renewcommand{\O}{\mathcal{O}}
\renewcommand{\char}{\textup{char}}

\title[Local heuristics for abelian varieties over finite fields]{Local heuristics and an exact formula for abelian varieties of odd prime dimension over finite fields}

\author{Jonathan Gerhard}
\address{James Madison University, Harrisonburg, VA 22807}
\email{gerha2jm@dukes.jmu.edu}

\author{Cassandra Williams}
\address{James Madison University, Harrisonburg, VA 22807}
\email{willi5cl@jmu.edu}
\urladdr{\url{http://educ.jmu.edu/~willi5cl}}


\begin{document}

\begin{abstract}
Consider a $q$-Weil polynomial $f$ of degree $2g$. Using an equidistribution assumption that is too strong to be true, we define and compute a product of local relative densities of matrices in $\GSp_{2g}(\Fell)$ with characteristic polynomial $f\mod\ell$ when $g$ is an odd prime.  This infinite product is closely related to a ratio of class numbers. When $g=3$ we conjecture that the product gives the size of an isogeny class of principally polarized abelian threefolds.
%
\end{abstract}

\maketitle
\setcounter{tocdepth}{1}
\tableofcontents

\section{Introduction}
This paper is a direct generalization of the work of Achter-Williams~\cite{achterwilliams2015} to abelian varieties of odd prime dimension, and is guided in philosophy by both Gekeler \cite{gekeler03} and Katz \cite{katz-lt}. We begin by considering abelian varieties over a finite field $\Fq$, where $q$ is a power of a prime. To each such variety $X$, we can associate a characteristic polynomial of Frobenius $f_X(T) \in \Z[T]$. A theorem of Tate~\cite{tate66} tells us that two varieties are isogenous if and only if their characteristic polynomials are equal.

Let $\A_g(k)$ denote the moduli space of principally polarized abelian varieties of dimension $g$ over a field $k$, where each variety is weighted by the size of its automorphism group.  Also define $\A_g(\Fq; f)$ to be those members of $\A_g(\Fq)$ with characteristic polynomial $f(T)$.  Then $\A_g(\Fq; f)$ denotes the set of isomorphism classes of principally polarized abelian varieties of dimension $g$ over $\Fq$ with characteristic polynomial of Frobenius $f$, weighted inversely by the size of the automorphism group, and computing $\#\A_g(\Fq;f)$ gives the number of (isomorphism classes of) abelian varieties over $\Fq$ in a particular isogeny class.

The Frobenius endomorphism gives an automorphism of the Tate module $T_{\ell}X$ (for $\ell$ not dividing $q$), so it has a representation as an element of the matrix group $\GSp_{2g}(\Fell)$. For each $\ell$, we define a term $\nu_{\ell}(f)$ measuring the relative frequency of $f \mod \ell$ as the characteristic polynomial for an element of $\GSp_{2g}(\Fell)$, as well as an archimedean term $\nu_{\infty}(f)$. 

Since (as much as possible) Frobenius elements are equidistributed in $\GSp_{2g}(\Fell)$, it seems possible that the product of these local densities could at least estimate the value of $\#\A_g(\Fq;f)$.  This, of course, is a ridiculous tactic, as the $\bmod$ $\ell$ Frobenius elements are only equidistributed when $q\gg_g\ell$.  However, as in \cite{achterwilliams2015}, we again show that this local data does  apparently control isogeny class size.

Our main result is as follows.  Consider a particular class of $q$-Weil polynomials $f$ of degree $2g$ (see the next section for details).  Let $K$ be the splitting field of $f$ over $\Q$ and $K^+$ be its maximal totally real subfield, with class numbers $h_K$ and $h_{K^+}$ respectively.  
A theorem of Everett Howe in the preprint \cite{howepreprint}, together with the conditions stated in the next section, implies that 
$$\#\A_{g}(\Fq; f)=\frac{h_K}{h_{K^+}}$$
(see Theorem~\ref{Thm:Howe} and Corollary~\ref{CorToHowe}).  Then an immediate corollary of this theorem and our work is that 
$$\nu_{\infty}(f)\prod_{\ell}\nu_{\ell}(f) = \#\A_{g}(\Fq; f)$$
when $g = 3$ (see Corollary \ref{Cor:g3UseHowe}). (For $g>3$ an odd prime, there is only a minimal obstruction to this corollary which is explained at the end of section \ref{SectionCentralizerOrders} and in Remark \ref{RemarkNonSSWillMatch}.)

\section{Abelian varieties and Weil polynomials}\label{Sec:WeilPolys}

Let $X/\F_q$ be an abelian variety of dimension $g$ over a finite field of $q=p^a$ elements, and let $f(T)\in\Z[T]$ be the characteristic polynomial of its Frobenius endomorphism.  Then $f(T)$ is a $q$-Weil polynomial, a polynomial of degree $2g$ with complex roots $\alpha_1,\dots,\alpha_{2g}$ with $\abs{\alpha_j}=\sqrt{q}$ for every $j$, where the ordering can be chosen so that $\alpha_j\alpha_{g+j}=q$ for $1\leq j \leq g$.

Each $q$-Weil polynomial $f(T)$ corresponds to a (possibly empty) isogeny class $\mathcal{I}_f$ of abelian varieties of dimension $g$ over $\F_q$.  Following \cite{achterwilliams2015}, we will assume:

\begin{enumerate}\def\theenumi{W.\arabic{enumi}}
\item ({\em ordinary}) \label{enord} the middle coefficient of $f$ is relatively prime to $p$;

\item ({\em principally polarizable}) \label{enpp} there exists a principally polarized abelian variety of dimension $g$ with characteristic polynomial $f$;

\item ({\em cyclic}) \label{encyclic}  the polynomial $f(T)$ is irreducible over $\Q$, and $K_f := \Q[T]/f(T)$ is Galois, cyclic, and unramified at $p$;

\item ({\em maximal}) \label{enmax} for $\pi_f$ a (complex) root of $f(T)$, with complex conjugate $\bar\pi_f$, then $\O_f := \Z[\pi_f,\bar\pi_f]$, {\em a priori} an order in $K_f$, is actually the maximal order $\O_{K_f}$.
\end{enumerate}

Assumptions \eqref{enord}, \eqref{enpp}, and \eqref{enmax} are identical to those in \cite{achterwilliams2015}.  

The condition \eqref{encyclic} is similar; we want to assume $K_f$ is abelian and Galois.  In \cite[(W.3)]{achterwilliams2015}, the authors only assumed that $K_f$ is Galois, but as $K_f$ was a number field of degree 4 this also guaranteed it to be abelian.  Many of the results proven in the present work require only that $K_f$ is abelian and Galois; however, we will assume that $g$ is an odd prime in many of our major results, and thus $K_f$ will be cyclic (with $\Gal(K_f/\Q)\cong \Z/2g\Z$).



\begin{rem}
It should be noted that \eqref{encyclic} is in fact a serious restriction, as number fields $K_f$ with a cyclic Galois group are quite rare among all number fields of degree $2g$.  Our method does work in broader contexts; for example, in \cite{rauchthesis} the author proves analogous results to those in \cite{achterwilliams2015} and our own for degree 4 fields with a nonabelian Galois group.  It seems likely that our methods readily generalize to any abelian Galois extension $K_f$, at the cost of more elaborate and extensive computations as the factorization of $g$ becomes more complex.

Thus, in this work, we will restrict to cyclic Galois groups in the hope of demonstrating our method and heuristic in the simplest generalized situation, rather than distracting the reader with details that provide no new insight into the problem.
\end{rem}

Note that $K_f$ is a CM field, and as such it comes equipped with an intrinsic complex conjugation $\iota\in\Gal(K_f/\Q)$.  Also, the isomorphism class of $\O_f$ (as an abstract order) is independent of the choice of $\pi_f$.

\begin{ex}
The polynomial $f(T)=T^6 + 10 T^5 + 48 T^4 + 151 T^3 + 336 T^2 + 490 T + 343$ is a 7-Weil polynomial that meets all of the assumptions \eqref{enord}-\eqref{enmax} when $g=3$.
\end{ex}

The Weil polynomial $f(T)$ factors as $f(T)=\prod_{j = 1}^{g} (T - \sqrt{q}e^{i\theta_j})(T - \sqrt{q}e^{-i\theta_j})$; then under our assumptions the polynomial $f^+(T)= \prod_{j = 1}^{g} (T - 2\sqrt{q}\cos(\theta_j))$ is the minimal polynomial of $\pi_f+\bar\pi_f$ and $K^+_f=\Q[T]/f^+(T)$ is the maximal totally real subfield of $K_f$.    Note that $\Z[\pi_f]\cong \Z[T]/f(T) \subset \O_f$ and define the conductor of $f$, $\cond(f)$ as the index $[\O_f:\Z[\pi_f]]$. We will denote the discriminants of the polynomials $f$ and $f^+$ as $\disc(f)$ and $\disc(f^+)$, respectively,  while $\Delta_\O$ will represent the discriminant of an order $\O$.  Note that $\Delta_{\Z[\pi_f]}=\disc(f)$ and $\Delta_{\O_{K^+_f}}=\disc(f^+)$.

In the following technical lemma, we give explicit forms for $\disc(f)$ and $\disc(f^+)$ for any positive integer $g$. 
\begin{lem}\label{discs} Let $f$ be a $q$-Weil polynomial of degree $2g$ with $g \geq 1$. 
Then
\begin{enumerate}
\item $\begin{displaystyle}\disc(f) = (-1)^g2^{2g^2}q^{2g^2 - g}\left(\prod_{j = 1}^g \sin^2(\theta_j)\right)\left(\prod_{1 \leq k < t \leq g} (\cos(\theta_k) - \cos(\theta_t))^2\right)^2 \end{displaystyle}$
\end{enumerate}
and
\begin{enumerate}
\setcounter{enumi}{1}
\item  $\begin{displaystyle} \disc(f^+) = 2^{g(g-1)}q^{\frac{g(g-1)}{2}}\prod_{1 \leq k < t \leq g} (\cos(\theta_k) - \cos(\theta_t))^2 \end{displaystyle}.$
\end{enumerate}
\end{lem}
\begin{proof}
Recall that the roots of $f(T)$ are of the form $\sqrt{q} e^{\pm i\theta_j}$ for $1\leq j\leq g$.  Then the proof of part (1) of the lemma proceeds by induction on $g$ using elementary methods and is omitted here.  A direct computation of the discriminant of $f^+$, which has roots $\sqrt{q}e^{i\theta_j} + \sqrt{q} e^{-i\theta_j}=2\sqrt{q}\cos(\theta_j)$ for $1\leq j \leq g$, proves part (2).
\end{proof}

\begin{rem}
See \cite[Theorem 4.3]{howepreprint} for a similar computation relating the Frobenius angles $\theta_i$ to (in our notation) 
$$\sqrt{\Delta_{\O_{K_f}}/\Delta_{\O_{K_f^+}}}.$$
\end{rem}

The explicit forms of Lemma \ref{discs}  will be helpful in proving the following lemma, as well as for defining local factors in section \ref{Sec:LocalFactorsf}.

\begin{lem}\label{lemconductor}
The index of $\Z[\pi_f]$ in $\O_f$ is $q^{\frac{g(g-1)}{2}}$.
\end{lem}
\begin{proof}
From \cite{howe95}, we have 
$\Delta_{\O_f} = (-1)^g\disc(f^+)^2 \, N_{K_f/\Q}(\pi_f - \bar{\pi}_f)$
and $\disc(f) = \cond(f)^2 \, \Delta_{\O_f}.$ Then 
$$\cond(f)^2 = (-1)^g\frac{\disc(f)}{\disc(f^+)^2 N_{K_f/\Q}(\pi_f - \bar{\pi}_f)}.$$


Without loss of generality, choose $\pi_f = \sqrt{q}e^{i\theta_1}$. Then $\pi_f - \bar{\pi}_f = 2i\sqrt{q}\sin(\theta_1)$ and all of its Galois conjugates are of the form $\pm 2i\sqrt{q}\sin(\theta_j)$ for $j \in \{1,2, \dots, g\}$. Therefore, 
$$N_{K_f/\Q}(\pi_f - \bar{\pi}_f) = \prod_{j = 1}^g (2i\sqrt{q}\sin(\theta_j))(-2i\sqrt{q}\sin(\theta_j)) = 4^g q^{g} \prod_{j=1}^g \sin^2(\theta_j).$$
Applying $\disc(f)$ and $\disc(f^+)$ from Lemma~\ref{discs}, we find $\cond(f)^2 = q^{g(g-1)}$, and the lemma follows.
\end{proof}

\begin{cor}
\label{lemalmostmaximal}
If $\ell\not = p$, then $\O_{K_f}\otimes \Z_{(\ell)} \cong \Z_{(\ell)}[T]/f(T)$.
\end{cor}

Corollary \ref{lemalmostmaximal} is proved, independent of the dimension of the abelian variety, in \cite{achterwilliams2015} and so its proof is omitted here.

Lastly, we prove that $\Z[T]/f^+(T)$ is the maximal order of $K_f^+$.

\begin{lem}
\label{lemmaximalrealorder}
The order $\Z[T]/f^+(T)$ is the maximal order $\O_{K_f^+}$.
\end{lem}
\begin{proof}
Condition \eqref{enmax} implies that $\O_f\cap K_f^+=\O_{K_f}\cap K_f^+=\O_{K_f^+}$.  Certainly $\Z[T]/f^+(T)=\Z[\pi_f+\bar\pi_f] \subseteq \O_f\cap K_f^+$.

Let $\alpha\in K_f^+=\Q[T]/f^+(T)=\Q(\pi_f+\bar\pi_f)$, which has dimension $g$ over $\Q$.  Then 
$$\alpha=a_0+a_1(\pi_f+\bar\pi_f) + a_2(\pi_f+\bar\pi_f)^2+\cdots+a_{g-1}(\pi_f+\bar\pi_f)^{g-1}$$
with all $a_i\in\Q$.  Suppose $\alpha$ is also an element of $\O_f$, so $\alpha\in\O_f\cap K^+_f$.  It is straightforward to show that the set 
$$\set{1, \pi_f, \pi_f^2,\ldots,\pi_f^g,\bar\pi_f,\bar\pi_f^2,\ldots,\bar\pi_f^{g-1}}$$
forms a basis for $\O_f=\Z[\pi_f,\bar\pi_f]$.  Recall that $\pi_f\bar\pi_f=q$; expand $\alpha$ and collect powers of $\pi_f$ and $\bar\pi_f$.  Notice that the coefficient of $\pi_f^{g-1}$ in $\alpha$ is exactly $a_{g-1}$, and so $a_{g-1}\in\Z$.   Using back substitution on the coefficients of powers of $\pi_f$ and $\bar\pi_f$, we find that $a_{g-2}\in\Z$, $a_{g-3}\in\Z$, and so on.  Therefore, all $a_i\in\Z$ and $\alpha\in\O_f\cap K_f^+$ is such that $\alpha\in\Z[\pi_f,\bar\pi_f]$ and $\O_f\cap K_f^+\subseteq \Z[\pi_f+\bar\pi_f].$  Therefore, $\Z[\pi_f+\bar\pi_f]=\O_{K_f^+}$.
%
%
\end{proof}

\section{Conjugacy classes in $\GSp_{2g}(\F_\ell)$}\label{Sec:cclasses}

\subsection{Symplectic groups and conjugacy}

The symplectic group $\GSp_{2g}(\Fell)$ is the subgroup of $\GL_{2g}(\Fell)$ preserving an antisymmetric bilinear form $J$ up to a scalar multiple. We choose  $$J = \begin{bmatrix}
0 & \I_g \\
-\I_g & 0
\end{bmatrix}$$ 
but note that different choices of $J$ produce isomorphic copies of $\GSp_{2g}(\Fell)$. 

Explicitly,  
$$\GSp_{2g}(\Fell) = \{ M \in \GL_{2g}(\Fell) \mid MJM^T = mJ \text{ for some $m \in \Fell^{\times}$}\}.$$ 
The value $m$ is called the \emph{multiplier} of $M$. All matrices in $\GSp_{2g}(\Fell)$ have the property that there exists a pairing (dictated by the choice of antisymmetric bilinear form) of its eigenvalues such that each pair has product $m$.

In \cite[Theorem~1.18]{shinoda80}, Shinoda parametrizes the set of conjugacy classes of $\GSp_{2g}(\Fell)$.  For our purposes, we do not need their parametrization in full generality, and will describe the relevant portions in our own notation. Let 
$$f(T) = T^d + c_{d-1}T^{d-1} + \dots + c_1T + c_0$$ 
be a polynomial in $\Fell[T]$.   Define the \emph{dual} of $f(T)$ with respect to the multiplier $m$ to be 
$$\bar{f}^m(T) = \frac{T^d}{c_0}f(mT^{-1}).$$ 
(We will occasionally omit the $m$ on the left hand side for notational convenience.) Then we have three types of polynomials:
\begin{enumerate}
\item \emph{Root polynomials} are polynomials of the form $f(T) = T^2 - m$ when $m$ is not a square, or either of $f(T) = T \pm \sqrt{m}$ when $m$ is a square.  (It is easy to check that all root polynomials satisfy $\bar{f}^m(T)=f(T).$)
\item \emph{$\alpha$ pairs} are pairs of polynomials $(f(T), \bar{f}^m(T))$ such that $ \bar{f}^m(T)\neq f(T)$.
\item \emph{$\beta$ polynomials} are polynomials $f(T)$ such that $\bar{f}^m(T)=f(T)$ and $f(T)$ is not a root polynomial.
\end{enumerate}


Note that any linear polynomial satisfying $\bar{f}^m(T)=f(T)$ is a root polynomial.  One consequence of these definitions is as follows.
\begin{lem}\label{NoOddBetas}
There are no irreducible $\beta$ polynomials of odd degree. That is, for all irreducible nonlinear polynomials $f(T)$ of odd degree and all multipliers $m\in\Fq^\times$, $$ \bar{f}^m(T)\neq f(T).$$
\end{lem}

\begin{proof}
Let $f(T) = T^d + c_{d-1}T^{d-1} + \dots + c_1T + c_0$ be an irreducible polynomial with $d\geq 3$ odd and suppose $\bar{f}^m(T)=f(T)$. This implies $c_0^2 = m^d$, but if $m$ is non-square then $m^d$ is non-square since $d$ is odd, which is a contradiction. If instead $m$ is a square, then $-\sqrt{m} \in \Fell$ is a root of $f(T)$, contradicting the fact that $f(T)$ is irreducible.
\end{proof}

The general theory of conjugacy classes in $\GL_n$ as well as the additional intricacies of conjugacy in $\GSp_n$ are given (briefly) in \cite[Section~3.2]{achterwilliams2015}, and the reader is encouraged to revisit this section if needed.  For our current purposes, recall that to each irreducible factor of the characteristic polynomial of a matrix, we associate a partition of its multiplicity; the characteristic polynomial together with its partition data determine conjugacy in $\GL_n$.  We also remind the reader that \emph{cyclic} matrices are those for which the characteristic and minimal polynomials coincide, and thus are those where all partitions are maximal (consist of only one part).  Then characteristic polynomials for conjugacy classes of cyclic matrices in $\GSp_{2g}(\Fell)$ are constructed as follows. 

\begin{thm}\label{ConjClassCharacterization} The following characteristic polynomials uniquely determine a conjugacy class of cyclic matrices in $\GSp_{2g}(\Fell)$, where, for each $m$, the first product is over all $\alpha$ pairs and the second product is over all $\beta$ polynomials.

\begin{enumerate}
\item For square $m \in \Fell^{\times}$,
 $$f(T) = (T - \sqrt{m}\,)^{e_{R_1}}(T + \sqrt{m}\,)^{e_{R_2}}\prod_{\alpha}(f_\alpha(T)\bar{f_\alpha}(T))^{e_\alpha} \prod_{\beta} (f_\beta(T))^{e_\beta}$$ 
 for a choice of $e_{R_1}, e_{R_2}, e_\alpha,$ and $e_\beta$ such that any odd parts in the partitions of $e_{R_1}$ and $e_{R_2}$ have even multiplicity and
 $$e_{R_1} + e_{R_2} + 2\sum_{\alpha}\deg(f_\alpha)e_\alpha + \sum_\beta \deg(f_\beta)e_\beta = 2g.$$

\item For non-square $m \in \Fell^{\times}$,
$$f(T) = (T^2 - m)^{e_{R}}\prod_{\alpha}(f_\alpha(T)\bar{f_\alpha}(T))^{e_\alpha} \prod_{\beta} (f_\beta(T))^{e_\beta}$$ 
for a choice of $e_R, e_\alpha,$ and $e_\beta$ such that any odd parts in the partition of $e_{R}$ have even multiplicity and
$$2e_R + 2\sum_{\alpha}\deg(f_\alpha)e_\alpha + \sum_\beta \deg(f_\beta)e_\beta = 2g.$$
\end{enumerate}
For each exponent $e_k$ defined above, the only allowable partition is $[e_k]$.
\end{thm}
\begin{rem}
Shinoda's parameterization (\cite[Theorem~1.18]{shinoda80}) also includes sets of (nondegenerate symmetric) bilinear forms with ranks relating to the number of parts of even sizes in the partitions of any root polynomials present in the factorization of the characteristic polynomial.  The number of equivalence classes of these bilinear forms detect when a characteristic polynomial (and associated set of partitions) gives rise to two conjugacy classes in $\GSp_{2g}(\F_\ell)$, indexed by $\set{+,-}$.   We omit this component of Shinoda's parameterization because, when $g$ is odd, all of the characteristic polynomials that we will declare as \emph{relevant} in the next section give rise to only one conjugacy class.  
\end{rem}

\subsection{Relevant conjugacy classes of $\GSp_{2g}(\F_\ell)$}\label{conjclasses}

For the remainder of section \ref{Sec:cclasses}, let $g$ be an odd prime.

We want to identify the possible shapes (factorization structures) of characteristic polynomials  which correspond to conjugacy classes of cyclic matrices in $\GSp_{2g}(\Fell)$ which respect the assumptions \eqref{enord}-\eqref{enmax}.  Therefore, a \emph{relevant} characteristic polynomial is one such that
\begin{itemize}
\item  the factorization is as in Theorem \ref{ConjClassCharacterization},
\item  the degrees of all irreducible factors are equal, and 
\item  the multiplicities of each irreducible factor are equal. 
\end{itemize}
The first condition forces our matrices to be cyclic elements of $\GSp_{2g}(\Fell)$, and the other two correspond with the requirement that $K_f$ be Galois from \eqref{encyclic}.

Let $[d]$ denote a monic irreducible degree $d$ polynomial in $\Fell[T]$, and assume that $[d]_i\neq[d]_j$ if $i\neq j$.  Then the following are the shapes  of all relevant characteristic polynomials.
\begin{center}
\def\arraystretch{1.3}
\begin{tabular}{ll}
$[1]_1\dots[1]_{2g}$ & $[1]^{2g}$\\
$[2]_1\dots[2]_g$ & $[1]_1^g[1]_2^g $\\
$[g]_1[g]_2$ & $[2]^g$\\
$[2g]$ & \\
\end{tabular}
\end{center}
We exclude the shape $[g]^2$ by Lemma~\ref{NoOddBetas}. Additionally, we exclude $[1]_1^2\dots[1]_g^2$; since there are an odd number of factors, it must be that  one of these linear factors is a root polynomial while the rest are $\alpha$ pairs. Then the Galois group of $K_f$ cannot act transitively on the roots of such an $f$.

The seven relevant conjugacy class shapes fall into two categories: regular semisimple and non-semisimple. A regular semisimple conjugacy class contains elements with a squarefree characteristic polynomial (and thus are cyclic by definition).  A class is not semisimple when the characteristic polynomial of its elements is not squarefree. 

We list the relevant conjugacy classes by the shape of their characteristic polynomial in Tables \ref{Tab:RegSSCharPolShape} (regular semisimple) and \ref{Tab:NonSSCharPolShape} (non-semisimple).  Each table also contains information about the multiplier $m$ and the type of the irreducible factors. 

\begin{table}[h]
\begin{tabular}{*3c}    \toprule
Char. Pol. Shape & Valid $m$ & Polynomial type \\
\hline
$[1]_1\dots[1]_{2g}$ & All $m$ & $\alpha$ pairs\\
$[2]_1\dots[2]_g $ & All $m$ & $\beta$ polynomials\\
$[g]_1[g]_2$ & All $m$ & $\alpha$ pair\\
$[2g]$ & All $m$ & $\beta$ polynomial \\
\bottomrule
\end{tabular}
\caption{Relevant characteristic polynomial shapes for regular semisimple conjugacy classes}
\label{Tab:RegSSCharPolShape}
\end{table}

\begin{table}[h]
\begin{tabular}{*3c}    \toprule
Char. Pol. Shape & Valid $m$ & Polynomial type \\
\hline
$[1]^{2g}$ & Square $m$ & Root polynomial\\
$[1]_1^g[1]_2^g $ & All $m$ & $\alpha$ pair\\
$[2]^g$ & All $m$ & $\beta$ polynomial \\
\bottomrule
\end{tabular}
\caption{Relevant characteristic polynomial shapes for non-semisimple conjugacy classes}
\label{Tab:NonSSCharPolShape}
\end{table}

\subsection{Centralizer orders}\label{SectionCentralizerOrders}

Denote by $\cclass$ a conjugacy class of matrices in $\GSp_{2g}(\F_\ell)$ with characteristic polynomial $f_\cclass(T)$.  For each relevant conjugacy class, we will find its order by instead finding the order of its centralizer. 
Let $\cent_{\GSp_{2g}(\Fell)}(\cclass)$ be the centralizer in $\GSp_{2g}(\Fell)$ of an element of $\cclass$.  Since we need only the size of the centralizer, the choice of this element is arbitrary.

\begin{lem}\label{RegCent} Let $\mathcal{C}$ be a regular semisimple conjugacy class with characteristic polynomial having one of the shapes listed in Table \ref{Tab:RegSSCharPolShape}.  Then we have 
$$\#\cent_{\GSp_{2g}(\Fell)}(\cclass) = \begin{cases}
(\ell-1)^{g + 1} & \text{ if $f_{\cclass}(T)$ has shape $[1]_1 \dots [1]_{2g}$}, \\
(\ell^2 - 1)(\ell+1)^{g-1} & \text{ if $f_{\cclass}(T)$ has shape $[2]_1\dots[2]_g$}, \\
(\ell^g - 1)(\ell-1) & \text{ if $f_{\cclass}(T)$ has shape $[g]_1[g]_2$}, \\
(\ell^g + 1)(\ell - 1) & \text{ if $f_{\cclass}(T)$ has shape $[2g]$}.
\end{cases}$$
\end{lem}


\begin{proof}
Since each class is regular and semisimple, their centralizers are tori.  For example, consider the case where $f_{\cclass}(T)$ has the shape $[2g]$.  Then $f_{\cclass}(T)$ has roots $\alpha,\alpha^\ell,\ldots,\alpha^{\ell^{2g-1}}$ in $\F^\times_{\ell^{2g}}$ in a single orbit under the action of Galois.  Since $\cclass\subseteq \GSp_{2g}(\Fell)$, it must be that the elements of $\cclass$ have multiplier $m = \alpha^{\ell^{i}}\alpha^{\ell^{g + i}}$ for $i \in \{0, 1, \dots, g - 1\}$.  Thus, elements of the centralizer of $\cclass$ are those where roots of the characteristic polynomial are elements of $\F^\times_{\ell^{2g}}$ with an $\F_{\ell^g}$-norm lying in $\F^\times_\ell$.  There are $\frac{\ell^{2g} - 1}{\ell^g - 1} = \ell^g + 1$ elements of $\F^\times_{\ell^{2g}}$ with such a norm for a fixed $m$. Since we have $\ell - 1$ choices for the multiplier, $\#\cent_{\GSp_{2g}(\Fell)}(\cclass) = (\ell^g + 1)(\ell - 1).$

The centralizer sizes for the other cases are computed similarly.
\end{proof}


In the case where $\cclass$ is not semisimple, the process for determining the order of its centralizer is significantly more challenging.  In these cases, we must construct an explicit matrix $\gamma$ which is a representative of $\cclass$, and verify that $\gamma$ has the correct characteristic polynomial, that $\gamma$ is cyclic, and that $\gamma\in\GSp_{2g}(\Fell)$.  Then we must find an explicit matrix $C$ which is a generic member of the centralizer of that $\gamma$ (also in $\GSp_{2g}(\Fell)$) and use it to count the number of possible elements of $\cent_{\GSp_{2g}(\Fell)}(\cclass)$.

For any particular $g$, this process is possible.  (As an example, the centralizer orders of the non-semisimple classes for $g=3$ are given in Proposition~\ref{NonSSCent}.)  However, we have not yet constructed representatives for all three non-semisimple classes for a general (odd prime) $g$ and thus do not have formulae for their centralizer orders.  We hope, in future work, to address this gap.

\section{Local factors for $f$}\label{Sec:LocalFactorsf}

In this section we define local factors $\nu_\ell(f)$ for each finite rational prime $\ell$ and one for the archimedean prime, $\nu_\infty(f)$.  For all $\ell\neq p$, this local factor is given by the density of elements of $\GSp_{2g}(\F_\ell)$ with a fixed multiplier and characteristic polynomial $f$ with respect to the ``average" frequency.  We also define $\nu_p(f)$ and $\nu_\infty(f)$ based on the same notions.  These definitions are in direct analogue with those of \cite{achterwilliams2015}, and are thus philosophically guided by \cite{gekeler03} as well.

\subsection{$\nu_\ell(f)$}

Suppose $\ell\neq p$ is a rational prime and consider a principally polarized abelian variety $X/\Fq$ of dimension $g$. The Frobenius endomorphism $\pi_{X/\Fq}$ of $X$ acts as an automorphism of $X_\ell$, and scales by a factor of $q$ the symplectic pairing on $X_\ell$ induced by the polarization.  Thus, we can consider $\pi_{X/\Fq}$ as an element of $\GSp_{2g}(\F_\ell)^{(q)}$ (the set of elements of $\GSp_{2g}(\F_\ell)$ with multiplier $q$).

There are $\ell^g$ possible characteristic polynomials for an element of $\GSp_{2g}(\F_\ell)^{(q)}$, and so the average frequency of a particular polynomial occurring as the characteristic polynomial of an element of the group (with respect to all such polynomials) is given by
$$\#\GSp_{2g}(\F_\ell)^{(q)} / \ell^g.$$
Then for primes $\ell$ unramified in $K_f$, we define $\nu_\ell(f)$ as 
\begin{equation}\label{definenuellf}
\nu_\ell(f) = \frac{\#\set{\gamma\in\GSp_{2g}(\F_\ell)^{(q)} \mid \charpol(\gamma)\equiv f \bmod \ell}}{\#\GSp_{2g}(\F_\ell)^{(q)}/\ell^g}.
\end{equation}
(See \eqref{redefinenuellf} for a definition for all $\ell\neq p$.)

\subsection{$\nu_p(f)$}

The definition of $\nu_p(f)$ is similar to but more intricate than \eqref{definenuellf}.  Under our assumptions, $X/\F_q$ is an ordinary abelian variety of (odd prime) dimension $g$ with characteristic polynomial of Frobenius 
$$f_X(T)=T^{2g} + c_1 T^{2g-1} + \dots +c_g T^g+qc_{g-1} T^{g-1} + \dots +q^{g-1} c_1 T + q^g.$$  
Thus, as in \cite{achterwilliams2015}, we have a canonical decomposition of the $p$-torsion group scheme into \'etale and toric components $X[p]\cong X[p]^{\text{et}}\oplus X[p]^{\text{tor}}$.  By ordinarity, $X[p]^{\text{et}}(\bar\F_q)\cong (\Z/p)^g$ and $(X[p]^{\text{tor}})^*(\bar\F_q)\cong (\Z/p)^g$, and the $q$-power Frobenius $\pi_{X/\Fq}$ acts invertibly on $X[p](\bar\F_q)$.  This action of $\pi_{X/\Fq}$ on both the \'etale and toric components of $X[p]$ has characteristic polynomial $g_X(T)=T^g + c_1 T^{g-1} + \dots +c_g \bmod p$, and must preserve the decomposition of $X[p]$.  Let $m_g$ be the multiplier of $g_X$ (so $m_g$ is a $g^{\text{th}}$ root of $c_g^2$).  Thus, we set 

\begin{equation}\label{definenupf}
\nu_p(f)=\frac{\#\set{\gamma\in\GSp_{2g}(\Fp)^{(m_g)} \mid \charpol(\gamma)\equiv (g_X)^2 \bmod p \text{ and $\gamma$ semisimple}}}{\#\GSp_{2g}(\Fp)^{(m_g)}/p^g}.
\end{equation}

%

\subsection{$\nu_\infty(f)$}

Lastly, we define an archimedean term, which is related to the Sato-Tate measure.  As stated in \cite{achterwilliams2015}, the Sato-Tate measure on abelian varieties conjecturally explains the distribution of Frobenius elements, and is a pushforward of Haar measure on the space of ``Frobenius angles", $0\leq \theta_1\leq \dots \leq \theta_g \leq \pi$.  The Weyl integration formula \cite[p218, 7.8B]{weyl} gives the Sato-Tate measure on abelian varieties of dimension $g$ explicitly as 
$$\mu_{ST}(\theta_1, \dots, \theta_g) = 2^{g^2}\left( \prod_{j < k} (\cos(\theta_j) - \cos(\theta_k))^2 \right)\prod_{i = 1}^g \left(\frac{1}{\pi}\sin^2(\theta_i)d\theta_i\right).$$

Fixing a particular $q$, the set of angles $\set{\theta_1, \dots ,\theta_g}$ gives rise to a $q$-Weil polynomial. We use the induced measure on the space of all such polynomials to define the archimedean term $\nu_\infty(f)$. To derive this induced measure, we first write the polynomial in terms of its roots and in terms of its coefficients as 
\begin{align*}
f(T) &= \prod_{j = 1}^g (T - \sqrt{q}e^{i\theta_j})(T - \sqrt{q}e^{-i\theta_j}) \\
&= T^{2g} + c_1 T^{2g - 1} + \dots + c_gT^g + c_{g - 1}qT^{g - 1} + \dots + c_1q^{g-1}T + q^g,
\end{align*}
and perform a change of variables.
Thus, we find our induced measure on the space of all $q$-Weil polynomials to be 
$$\mu(c_1, \dots, c_g) = \frac{1}{q^{g^2}(2\pi)^g}\sqrt{\abs{\frac{\disc(f)}{\disc(f^+)}}} \, dc_1 \dots dc_g.$$

Note that there are approximately $q^{\dim \mathcal{A}_g} = q^{\frac{g(g+1)}{2}}$ principally polarized abelian varieties over $\Fq$, so $q^{\frac{g(g+1)}{2}}\mu(c_1, \dots, c_g)$ can be thought of as an archimedean predictor for $\#\A_g(\Fq; f)$. Then we define
\begin{equation}\label{nuinfinity}
\nu_\infty(f) = \frac{1}{\cond(f) (2\pi)^g} \sqrt{\abs{\frac{\disc(f)}{\disc(f^+)}}}.
\end{equation} 
(We note that definition \eqref{nuinfinity} holds for any $g$, not just odd prime $g$.)

\section{Polynomials and primes in $K$}

Fix a $q$-Weil polynomial $f(T)$ which satisfies conditions \eqref{enord}-\eqref{enmax}.  For the remainder of the paper, we write $K$ for $K_f$, $\O_K$ for $\O_f$, $K^+$ for $K^+_f$, $\Delta_K$ for $\Delta_{\O_K}$, and $\Delta_{K^+}$ for $\Delta_{\O_{K^+}}$.  Let $\kappa_\ell=\O_K\otimes\F_\ell$, a $2g$-dimensional vector space over $\F_\ell$.

Our goal in this section is to relate the polynomial $f(T) \bmod \ell$ to (a representative of) one of the conjugacy classes defined in section \ref{conjclasses}.  There are two lenses through which we can consider such a correspondence, as outlined in \cite[Section 5]{achterwilliams2015}.
%
%
%
Regardless of the perspective, we will use the factorization of $f(T)\bmod\ell$ to determine a cyclic element of $\GSp_{2g}(\F_\ell)$ whose semisimplification is conjugate to $\gamma_\ell$, the image of the action of $\pi_f$ on $\kappa_\ell$.  Then we define
\begin{equation}\label{redefinenuellf}
\nu_\ell(f)= \frac{\#\set{\gamma\in\GSp_{2g}(\F_\ell) \mid \gamma \text{ is cyclic with semisimplification } \gamma_\ell}}{\#\GSp_{2g}(\Fell)^{(q)}/\ell^g}.
\end{equation}

\begin{lem}
If $\ell\nmid p\Delta_K$, then definitions \eqref{definenuellf} and \eqref{redefinenuellf} coincide.
\end{lem}
\begin{proof}
If $\ell\nmid p\Delta_K$ then $\ell\nmid\disc(f)$ and so $f(T)\bmod\ell$ has distinct roots.  Under condition \eqref{encyclic}, any factorization of $f(T)\bmod\ell$ with distinct roots appears in Table \ref{Tab:RegSSCharPolShape}, and so any element with characteristic polynomial $f(T)\bmod\ell$ is conjugate to $\gamma_\ell$.  All regular semisimple elements are cyclic, so the lemma is proven.
\end{proof}

Note that $\charpol(\gamma_\ell)$ is precisely $f(T)\bmod\ell$.  Also note that $\kappa_\ell=\O_K/\ell\cong \F_\ell[T]/f(T)$ by Corollary \ref{lemalmostmaximal}, so the factorization of $f(T)\bmod\ell$ is determined by the splitting of $\ell$ in $\O_K$.  That is, if $f(T)\mod\ell=\prod_{1\leq j\leq r} g_j(T)^{e_j}$, then $\ell=\prod_{1\leq j\leq r} \lambda_j^{e_j}$ for primes $\lambda_j$ of $\O_K$ where the residue degree of $\lambda_j$ equals the degree of the irreducible polynomial $g_j(T)$.  

Because of Condition \eqref{encyclic}, $K/\Q$ is a finite Galois extension with $\Gal(K/\Q)\cong \Z/2g\Z$.  Then the residue degrees of the $\lambda_j$ are all equal to a common value $\frak{f}$, and the ramification degrees of the $\lambda_j$ are all equal to a common value $e$.  In particular $2g=e\frak{f}r$.  (Notice that this restriction is precisely how we identified relevant class shapes in Section \ref{conjclasses}.)  Without loss of generality, let $\lambda$ be a prime of $\O_K$ over $\ell$.  Let $D(\ell)$ and $I(\ell)$ denote the decomposition and inertia groups, respectively, of a rational prime $\ell$.  Lastly, if $\ell\nmid \disc(f)$, then $\Gal(\kappa(\lambda)/\F_\ell)$ (the residue field of $\lambda$) is cyclic.  Let $\Frob_K(\ell)\in\Gal(K/\Q)$ be the element which induces the generator of this group, and call it the Frobenius endomorphism of $\lambda$ over $\ell$.

Let $\Gal(K/\Q)=\gen{\sigma}$ so that complex conjugation is given by $\iota=\sigma^g$.  We classify the splitting of rational primes of $K$ by enumerating the possibilities for $D(\ell)$ and $I(\ell)$.  

\begin{lem}\label{lemFrob}
Suppose $f$ satisfies Conditions \eqref{enord}-\eqref{enmax}.  Let $\ell\neq p$ be a rational prime.  The cyclic shape of $\gamma_\ell$ is determined by the decomposition and inertia groups $D(\ell)$ and $I(\ell)$ as in Table \ref{Tab:Frob}.
\end{lem}

\begin{center}
\begin{table}[h]
\begin{tabular}{*5l}    \toprule
$D(\ell)$ & $I(\ell)$ & $\Frob_K(\ell)$ & $(e,\frak{f},r)$ & Class shape  \\\midrule
$\set{1}$ &$\set{1}$ & $1$ & $(1, 1, 2g)$ & $[1]_1\dots[1]_{2g}$ \\
$\gen{\sigma^g}$ &$\set{1}$ & $\sigma^g$ & $(1,2,g)$ & $[2]_1\dots[2]_g$ \\  
$\gen{ \sigma^2 }$ &$\set{ 1 }$ & $\sigma^2$ & $(1, g, 2)$ & $[g]_1[g]_2$ \\
$\gen{ \sigma^2 }$ & $\gen{ \sigma^2 }$ &  - & $(g, 1, 2)$ & $[1]_1^g[1]_2^g$\\
$\gen{ \sigma }$ & $\set{ 1 }$ & $\sigma^i$ for $(i,g)=1$  & $(1, 2g, 1)$ & $[2g]$\\
$\gen{ \sigma }$ & $\gen{ \sigma^2 }$ & - & $(g,2,1)$ & $[2]^g$\\
$\gen{ \sigma } $ & $\gen{ \sigma }$ & - & $(2g, 1, 1)$ & $[1]^{2g}$\\\bottomrule
\end{tabular}
\caption{Prime factorizations and conjugacy class shapes for $K$.}
\label{Tab:Frob}
\end{table}
\end{center}

Note that in every case, the data $D(\ell)$ and $I(\ell)$ determines a unique conjugacy class from those given in Tables \ref{Tab:RegSSCharPolShape} and \ref{Tab:NonSSCharPolShape}.

\begin{proof}
In Table \ref{Tab:Frob}, we enumerated all possibilities for pairs of subgroups $I(\ell)\subseteq D(\ell) \subseteq \Gal(K/\Q)$.  In every case, there are $r=\#\Gal(K/\Q)/\#D(\ell)$ distinct irreducible factors of $f(T)\bmod\ell$, each with degree $\frak{f}=\#D(\ell)/\#I(\ell)$ and multiplicity $e=\#I(\ell)$.  In all cases, this factorization pattern exactly determines the conjugacy class for which $\gamma_\ell$ is a representative.
\end{proof}

\section{Local factors for $K$}\label{Sec:LocalFactorsK}

Recall that $K/\Q$ is a finite Galois extension with $\Gal(K/\Q)\cong \Z/2g\Z$.  Then $K$ has two unique subfields, the maximal totally real field $K^+$ of index 2 in $K$, and a complex quadratic extension of $\Q$ which we will call $K_1$ in what follows.

Let $X(K)$ be the character group of the Galois group of $K$. For $\chi \in X(K)$, let $K^\chi$ be the fixed field of $\ker(\chi)$. For a rational prime $\ell$, define
$$\chi(\ell) = 
\begin{cases}
\chi(\Frob_{K^\chi}(\ell)) & \text{if $\ell$ is unramified in $K^\chi$,} \\
0 & \text{otherwise.}
\end{cases}$$

Let $S(K) = X(K) \smallsetminus X(K^+)$ and define 
\begin{equation}\label{Eqn:definenuellK}
\nu_{\ell}(K) = \prod_{\chi \in S(K)}\left(1-\frac{\chi(\ell)}{\ell}\right)^{-1}.
\end{equation}
Recall that $\sigma$ is a generator of $\Gal(K/\Q)$ and let $\chi$ be a generator of $X(K)$. Then $\langle \sigma^2 \rangle = \Gal(K^+/\Q)$ and $\langle \chi^2 \rangle = X(K^+)$, so $S(K) = \{\chi, \chi^3, \dots, \chi^{2g - 1}\}.$ A quick computation shows $K^{\chi^g} = K_1$ and $K^{\chi^i} = K$ for all other $\chi^i \in S(K)$. We have $\chi^{i}(\ell) = (\chi(\ell))^i$ for all odd $i \neq g$, and so to compute $\chi^i(\ell)$ for odd $i$, we only need to know $\chi(\ell)$ and $\chi^g(\ell).$ 

\begin{lem}
The values of $\chi(\ell)$ and $\chi^g(\ell)$ are determined by $D(\ell)$ and $I(\ell)$, as given in Table~\ref{Tab:Chars}.
\end{lem}
\begin{proof}
The values in the table follow from the definitions of $\chi$ and $\chi^g$ above, the Frobenius elements given in Table \ref{Tab:Frob}, and the fact that $\Gal(K/\Q)$ and $X(K)$ are cyclic.  Then when $\Frob_K(\ell)$ generates $\Gal(K/\Q)$, $\chi(\ell)$ is a primitive $2g^{th}$ root of unity.  In particular, the values in Table~\ref{Tab:Chars} are independent of the choice of generator for each of $D(\ell)$, $I(\ell)$, and $X(K)$.
\end{proof}

\begin{table}[h]
\begin{tabular}{*4l}    \toprule
$D(\ell)$ & $I(\ell)$ & $\{\chi(\ell),\chi^g(\ell)\}$ & Class shape  \\\midrule
$\set{1}$ &$\set{1}$ & $\{1,1\}$ & $[1]_1\dots[1]_{2g}$ \\
$\gen{\sigma^g}$ &$\set{1}$ & $\{-1,-1\}$ & $[2]_1\dots[2]_g$ \\  
$\gen{\sigma^2}$ &$\set{1}$ & $\{e^{2\pi i/g},1\}$ & $[g]_1[g]_2$ \\
$\gen{\sigma^2}$ & $\gen{\sigma^2}$ & $\{0 ,1\}$ & $[1]_1^g[1]_2^g$\\
$\gen{\sigma}$ & $\set{1}$ & $\{e^{\pi i/g} ,-1\}$ & $[2g]$\\
$\gen{\sigma}$ & $\gen{\sigma^2}$ & $\{0 ,-1\}$ & $[2]^g$\\
$\gen{\sigma}$ & $\gen{\sigma}$ & $\{0,0\}$ & $[1]^{2g}$\\\bottomrule
\end{tabular}
\vspace{2mm}
\caption{Values of imaginary characters for $K$.}
\label{Tab:Chars}
\end{table}

\section{Matching}

In this section, we will prove a series of propositions to establish equalities between the local factor defined for $f$ in Section \ref{Sec:LocalFactorsf} and the local term intrinsic to $K = K_f$ defined in Section \ref{Sec:LocalFactorsK} both for general odd prime $g$ and for the specific case when $g=3$.

\subsection{General case}

Let $g$ be an odd prime.  

In Proposition \ref{PropMatchGeneralCaseForEll}, we must restrict to regular semisimple conjugacy classes because we only have $\#\cent_{\GSp_{2g}(\Fell)}(\cclass)$ for all odd prime $g$ in those cases.

\begin{prop}\label{PropMatchGeneralCaseForEll}
Suppose $f$ is a $q$-Weil polynomial of degree $2g$ such that $f\mod\ell=f_\cclass$ for one of the conjugacy class shapes in Table  \ref{Tab:RegSSCharPolShape}.  If $\ell\neq p$ then $\nu_\ell(f)=\nu_\ell(K)$.
\end{prop}
\begin{proof}
Let $\cclass$ be a conjugacy class from Table \ref{Tab:RegSSCharPolShape}. Using the orbit-stabilizer theorem, we can rewrite $\nu_\ell(f)$ in terms of $\#\cent_{\GSp_{2g}(\Fell)}(\cclass)$ as
\begin{equation*}
\nu_\ell(f) = \frac{\#\cclass}{\#\GSp_{2g}(\F_\ell)^{(q)}/\ell^g} =  \frac{\#\GSp_{2g}(\Fell)/\#\cent_{\GSp_{2g}(\Fell)}(\cclass)}{\#\GSp_{2g}(\F_\ell)^{(q)}/\ell^g} = \frac{\ell^g(\ell - 1)}{\#\cent_{\GSp_{2g}(\Fell)}(\cclass)}.
\end{equation*}
Since we found the centralizer orders for each regular semisimple class in Theorem~\ref{RegCent}, we use this expression for $\nu_\ell(f)$ to calculate the third column of Table~\ref{Tab:Match}. 

Additionally, from \eqref{Eqn:definenuellK} we have
\begin{equation*}\nu_{\ell}(K) = \prod_{\chi \in S(K)}\left(1-\frac{\chi(\ell)}{\ell}\right)^{-1} = \left(\frac{\ell}{\ell - \chi(\ell)} \right) \left(\frac{\ell}{\ell - \chi^3(\ell)} \right) \cdots \left(\frac{\ell}{\ell - \chi^{2g-1}(\ell)} \right).
\end{equation*}
The fourth column of Table~\ref{Tab:Match} contains the relevant values from Table~\ref{Tab:Chars} from which we compute $\nu_\ell(K)$ in the fifth column. 

To compute the third and fourth rows of the fifth column of Table~\ref{Tab:Match}, recall that if $z$ is a primitive $(2g)^{th}$ root of unity, then $z^2$ is a primitive $g^{th}$ root of unity. Then since $(\ell^{2g} - 1) = (\ell^g - 1)(\ell^g + 1)$,
$$\ell^{2g} - 1 = \prod_{j = 1}^{2g} (\ell - z^j), \ \text{and} \ \ell^g - 1 = \prod_{j = 1}^{g} (\ell - z^{2j}),$$ 
we must have $$(\ell^g + 1) = \prod_{j = 1}^{g} (\ell - z^{2j - 1}).$$
We see the third and fifth columns of Table~\ref{Tab:Match} match, so the proposition is proved.
\end{proof}

\begin{center}
\begin{table}[h]
\begin{tabular}{*5c}    \toprule
Class Shape & $\#\cent(\cclass)$ & $\nu_{\ell}(f)$  &  $\{\chi(\ell),\chi^g(\ell)\}$ & $\nu_\ell(K)$  \\  \midrule
$[1]_1\dots[1]_{2g}$ & $(\ell-1)^{g + 1}$ & $\frac{\ell^g}{(\ell-1)^g}$ & $\{1,1\}$ &  $\left(\frac{\ell}{\ell - 1}\right)^g$ \\
$[2]_1\dots[2]_g $ & $(\ell^2 - 1)(\ell+1)^{g-1}$ & $\frac{\ell^g}{(\ell + 1)^g}$ & $\{-1,-1\}$ & $\left(\frac{\ell}{\ell + 1}\right)^g$  \\  
$[g]_1[g]_2$ & $(\ell^g - 1)(\ell-1)$ & $\frac{\ell^g}{\ell^g - 1}$ & $\{e^{2\pi i/g}, 1\}$ &  $ \frac{\ell^g}{\ell^g - 1}$  \\
 $[2g]$ & $(\ell^g + 1)(\ell - 1)$ & $\frac{\ell^g}{\ell^g + 1}$ & $\{e^{\pi i/g}, -1\}$ & $\frac{\ell^g}{\ell^g + 1}$  \\  \bottomrule
\end{tabular}
\hspace{2mm}
\caption{Evaluating $\nu_{\ell}(f)$ and $\nu_\ell(K)$ for regular semisimple classes.}
\label{Tab:Match}
\end{table}
\end{center}

In order to show that $\nu_p(f)=\nu_p(K)$, we first must determine the possible factorizations of $p$ in $\O_K$.  

\begin{lem}\label{lemprimesplitsKplustoK}
Let $K/\Q$ be a degree $2g$ CM number field, and let $X$ be an ordinary abelian variety of dimension $g$ over $\Fq$.  Suppose $K$ acts on $X$.  Then any prime of $K^+$ over $p$ splits in $K$.
\end{lem}
\begin{proof}
Recall that $K^+=\Q(\pi+\bar{\pi})$, and that $\pi,\bar\pi$ are roots of $T^2-(\pi+\bar\pi)T+q$ over $K$.  Let $\frak{p}^+$ be a prime of $K^+$ over $p$; since $\pi\bar\pi=q$, at least one of $\pi$ or $\bar\pi$ lies in $\frak{p}$, a prime of $K$ over $\frak{p}^+$.  However, if $\pi+\bar\pi \in \frak{p}^+$, then both $\pi$ and $\bar\pi$ are in $\frak{p}$.  This contradicts the assumption that $f$ was ordinary, so $\pi+\bar\pi \notin \frak{p}^+$.  Then 
\begin{align*}
T^2-(\pi+\bar\pi)T+q &\equiv T^2 - uT \bmod \frak{p}^+ \\
&\equiv T(T-u) \bmod \frak{p}^+
\end{align*}
where $u \in \bmod \frak{p}^+$ is nonzero.  Therefore, $\frak{p}^+$ splits in $K$.
\end{proof}

%
%
%

\begin{prop}\label{PropMatchGeneralCaseForP}
Suppose $f$ is as in Proposition \ref{PropMatchGeneralCaseForEll}.  Then $\nu_p(f)=\nu_p(K)$.
\end{prop}
\begin{proof}
We assumed in \eqref{encyclic} that $p$ is unramified in $K$ and $\Gal(K/\Q)=\Z/2g\Z$, and by Lemma \ref{lemprimesplitsKplustoK} all primes of $K^+$ over $p$ split in $K$.  Then the only possible factorizations of $p$ in $\O_K$ are that $p$ splits completely, or that $p\O_K=\mathfrak{p}_1 \mathfrak{p}_2$.

If $p$ splits completely, then $g_X(T)$ (from \eqref{definenupf}) factors as a product of linear polynomials over $\Fp$ and the set of \emph{semisimple} matrices with characteristic polynomial $(g_X(T))^2$ has the same cardinality as $[1]_1\dots [1]_{2g}$.  Then $\nu_p(f)=\nu_p(K)$ by the first row of Table \ref{Tab:Match}.

If instead $p\O_K= \mathfrak{p}_1\mathfrak{p}_2$ (so $p$ is inert in $\O_{K^+}$), then $g_X(T)$ is irreducible so the set of \emph{semisimple} matrices with characteristic polynomial $(g_X(T))^2$ has the same cardinality as $[g]_1[g]_2$.  Then $\nu_p(f)=\nu_p(K)$ by the third row of Table \ref{Tab:Match}.
\end{proof}

The following proposition is true for any value of $g$.

\begin{prop}
Let $f$ be a $q$-Weil polynomial of degree $2g$ with splitting field $K$ and $\O_K=\Z[\pi,\bar\pi]$. Then
$$\nu_{\infty}(f) = \frac{1}{(2\pi)^g}\sqrt{\abs{\frac{\Delta_K}{\Delta_{K^+}}}}$$
\end{prop}
\begin{proof}
This follows from the fact that $\disc(f) = \cond(f)^2\Delta_K$ and $\disc(f^+) = \Delta_{K^+}$.
\end{proof}

\subsection{Specific case ($g=3$)}

In the particular case when $g=3$, we can also prove that $\nu_\ell(f)$ and $\nu_\ell(K)$ match in the case when $f\bmod\ell = f_\cclass$ for a non-semisimple conjugacy class $\cclass$.  We begin by giving the orders of the centralizers for the non-semisimple conjugacy classes in $\GSp_6(\Fell)$.

\begin{prop}\label{NonSSCent} Let $\mathcal{C}$ be one of the non-semisimple conjugacy classes of matrices in $\GSp_6(\Fell)$ with characteristic polynomial in one of the shapes listed in Table \ref{Tab:NonSSCharPolShape}.  Then we have 

$$\#\cent_{\GSp_{6}(\Fell)}(\cclass) = \begin{cases}
\ell^3(\ell - 1) & \text{ if $f_{\cclass}(T)$ has shape $[1]^6$}, \\
\ell^2(\ell - 1)^2 & \text{ if $f_{\cclass}(T)$ has shape $[1]_1^3[1]_2^3$}, \\
\ell^2(\ell^2 - 1) & \text{ if $f_{\cclass}(T)$ has shape $[2]^3$}. \\
\end{cases}$$
\end{prop}

\begin{proof}
Refer back to the end of section~\ref{Sec:cclasses} for an outline of the general method for determining centralizer orders for non-semisimple classes. We demonstrate this method for $g = 3$ in the case where $f_\cclass(T)$ has the shape $[1]_1^3[1]_2^3.$

Suppose $f_\cclass(T) = (T - a)^3(T - b)^3$ with $a,b \in \Fq^{\times}.$ Since these are $\alpha$ polynomials (by Table~\ref{Tab:NonSSCharPolShape}), $m = ab$. A representative of the class $\cclass$ is 
$$\gamma = \begin{bmatrix}
a & -a & 0 & 0 & 0 & 0 \\
 0 & a & -a & 0 & 0 & 0 \\
 0 & 0 & a & 0 & 0 & 0 \\
 0 & 0 & 0 & b & b & b \\
 0 & 0 & 0 & 0 & b & b\\
  0 & 0 & 0 & 0 & 0 & b\\
\end{bmatrix},$$
where we verify that $\charpol(\gamma)=f_\cclass(T)$, $\gamma\in\GSp_6(\Fell)$ with multiplier $ab$, and $\minpol(\gamma)=\charpol(\gamma)$ (so that $\gamma$ is cyclic).  Then a generic element of $\cent_{\GSp_{6}(\Fell)}(\cclass)$ has the form
$$C = \begin{bmatrix}
c_1 & y_1 & y_2 & 0 & 0 & 0 \\
 0 & c_1 & y_1 & 0 & 0 & 0 \\
 0 & 0 & c_1 & 0 & 0 & 0 \\
 0 & 0 & 0 & c_2 & c_3 & c_4 \\
 0 & 0 & 0 & 0 & c_2 & c_3\\
  0 & 0 & 0 & 0 & 0 & c_2\\
\end{bmatrix}$$
where  $$y_1 = -\frac{c_1c_3}{c_2} \hspace{5mm} \text{and} \hspace{5mm}  y_2 = c_1\left(\frac{c_3^2 - c_2c_4}{c_2^2}\right).$$
Since $\det C=c_1^3 c_2^3$, it must be that $c_1,c_2\in\Fell^\times$, while $c_3,c_4\in\Fell$.  Then $$\#\cent_{\GSp_{6}(\Fell)}(\cclass) = \ell^2(\ell-1)^2.$$

The other cases require similar computations.
\end{proof}

\begin{prop}\label{PropMatchSpecificCaseForEll}
Suppose $f$ is a $q$-Weil polynomial of degree $6$ such that $f\bmod\ell=f_\cclass$ for one of the conjugacy classes in Table \ref{Tab:NonSSCharPolShape}.  If $\ell\neq p$ then $\nu_\ell(f)=\nu_\ell(K)$.
\end{prop}
\begin{proof}
This proof is identical to the proof of Proposition~\ref{PropMatchGeneralCaseForEll}, but for non-semisimple classes. The second column of Table~\ref{Tab:NonSSMatch} comes from Proposition~\ref{NonSSCent}, which we use to compute the third column.

The fourth column comes from Table~\ref{Tab:Chars}, which we use to compute the fifth column. Seeing that the third and fifth columns match, we are done.
\end{proof}

\begin{center}
\begin{table}[h]
\begin{tabular}{*5c}    \toprule
Class Shape & $\#\cent(\cclass)$ & $\nu_{\ell}(f)$  &  $\{\chi(\ell),\chi^g(\ell)\}$ & $\nu_\ell(K)$  \\  \midrule
$[1]^6$ & $\ell^3(\ell - 1)$ & 1 & $\{0,0\}$ &  $\left(\frac{\ell}{\ell - 1}\right)^g$ \\
$[1]_1^3[1]_2^3$ & $\ell^2(\ell-1)^{2}$ & $\frac{\ell}{\ell - 1}$ & $\{0,1\}$ & $\frac{\ell}{\ell - 1}$  \\  
$[2]^3$ & $\ell^2(\ell^2-1)$ & $\frac{\ell}{\ell + 1}$ & $\{0,-1\}$ &  $\frac{\ell}{\ell + 1}$   \\  \bottomrule
\end{tabular}
\hspace{2mm}
\caption{Evaluating $\nu_{\ell}(f)$ and $\nu_\ell(K)$ for non-semisimple classes when $g=3$.}
\label{Tab:NonSSMatch}
\end{table}
\end{center}

\section{Main results}

This section generalizes many of the results of \cite[Section 7]{achterwilliams2015}.  We define an infinite product of numbers $\{a_\ell\}$ indexed by finite primes by 
\begin{equation*}
\prod_{\ell} a_\ell = \lim_{B \rightarrow \infty} \prod_{\ell < B}a _\ell
\end{equation*}
so that $\prod_\ell a_\ell \prod_\ell b_\ell = \prod_\ell (a_\ell b_\ell)$.

For a number field $L$, let $h_L$, $\omega_L$, and $R_L$ denote the class number, number of roots of unity, and regulator of $L$ respectively.

\begin{prop}\label{PropClassNumberRatio}
Let $f$ be a degree $2g$ $q$-Weil polynomial that is ordinary, principally polarizable, cyclic, and maximal. Let $K$ be the splitting field of $f$ over $\Q$ and let $K^+$ be its maximal totally real subfield. Then
\begin{equation}\label{Eq:classnumberratio}
\frac{h_K}{h_{K^+}} = \omega_K \nu_\infty(f) \prod_\ell \nu_\ell(K).
\end{equation}
\end{prop}
\begin{proof}
By the analytic class number formula, the ratio of class numbers on the left side of \eqref{Eq:classnumberratio} is $$\frac{h_K}{h_{K^+}} = \lim_{s \rightarrow 1} \frac{(s-1)\zeta_K(s)}{(s-1)\zeta_{K^+}(s)}\frac{\sqrt{\abs{\Delta_K}}2^g\omega_K R_{K^+}}{\sqrt{\abs{\Delta_{K^+}}}(2\pi)^g \omega_{K^+} R_{K}}.$$
For a finite abelian extension $L/\Q$, we have $$\lim_{s \rightarrow 1}(s-1)\zeta_L(s) = \prod_{\chi \in \Gal(L/\Q)^*\backslash \text{id}} L(1, \chi)$$ where we interpret the Dirichlet $L$-function as the conditionally convergent Euler product
$$L(1, \chi) = \lim_{B \rightarrow \infty} \prod_{\ell < B} \frac{1}{1 - \chi(\ell)/\ell}.$$

We then see that 
\begin{align*}
\lim_{s \rightarrow 1}\frac{ (s - 1)\zeta_K(s)}{(s - 1)\zeta_{K^+}(s)} &= \frac{\prod_{\chi \in \Gal(K/\Q)^*\backslash \text{id}}\left(\prod_{\ell} \frac{1}{1 - \chi(\ell)/\ell}\right)}{\prod_{\chi \in \Gal(K^+/\Q)^*\backslash \text{id}}\left(\prod_{\ell} \frac{1}{1 - \chi(\ell)/\ell}\right)}\\
&= \prod_{\chi \in S(K)} \prod_{\ell}\frac{1}{1 - \chi(\ell)/\ell}\\
&= \prod_{\ell} \nu_{\ell}(K)
\end{align*}
where $S(K)$ is as in Section~\ref{Sec:LocalFactorsK}.

By \cite[Proposition 4.16]{Wash97} $R_K = \frac{1}{Q}2^{g-1}R_{K^+}$, where $Q$ is Hasse's unit index. Since $K/\Q$ is cyclic, $Q=1$ by {\cite[Theorem 3]{furuya_1977}}. Therefore, $R_K=2^{g-1} R_{K^+}$.

Lastly, $\omega_{K^+} = 2$ since $K^+$ is totally real, so
\begin{align*}
\frac{h_K}{h_{K^+}} &= \sqrt{\abs{\frac{\Delta_K}{\Delta_{K^+}}}}\frac{\omega_K R_{K^+}}{2\pi^g R_K}\prod_{\ell}\nu_{\ell}(K) \\
&= \sqrt{\abs{\frac{\Delta_K}{\Delta_{K^+}}}}\frac{\omega_K}{(2\pi)^g}\prod_{\ell}\nu_{\ell}(K) \\
&= {\omega_K} \nu_\infty(f) \prod_{\ell}\nu_{\ell}(K).
\end{align*}
\end{proof}

\begin{rem}\label{RemarkNonSSWillMatch}
Propositions \ref{PropMatchGeneralCaseForEll} and \ref{PropMatchGeneralCaseForP} tell us that in many cases the $\nu_\ell(K)$ in the previous proposition is in fact equal to $\nu_\ell(f)$.  While we did not give the orders of the centralizers of the non-semisimple conjugacy classes for all odd prime $g$ in this paper, we have partial progress which suggests that Proposition \ref{PropMatchSpecificCaseForEll} will also generalize to all odd prime $g$.  On the assumption that Proposition \ref{PropMatchSpecificCaseForEll} in fact generalizes to all odd prime $g$, we make the following conjecture.
\end{rem}

\begin{conj}\label{ConjMainThmGeneral}
Let $g$ be an odd prime, and let $f$ be a degree $2g$ $q$-Weil polynomial that is ordinary, principally polarizable, cyclic, and maximal. Let $K$ be the splitting field of $f$ over $\Q$ and let $K^+$ be its maximal totally real subfield. Then
\begin{equation}
\nu_\infty(f)\prod_{\ell}\nu_\ell(f) = \frac{1}{\omega_K}\frac{h_{K}}{h_{K^+}}.
\end{equation}
\end{conj}
\begin{proof}
For any primes $\ell$ where $f_\cclass\cong f \bmod\ell$ has one of the regular semisimple shapes in Table \ref{Tab:RegSSCharPolShape}, Propositions \ref{PropMatchGeneralCaseForEll} and \ref{PropMatchGeneralCaseForP} give that $\nu_\ell(f)=\nu_\ell(K)$.  

Assume that $\nu_\ell(f)=\nu_\ell(K)$ for primes $\ell$ where $f_\cclass\cong f \bmod\ell$ is non-semisimple for all odd prime $g$. (That is, assume a version of Proposition \ref{PropMatchSpecificCaseForEll}  is true for all odd prime $g$.) Then for all primes $\ell$ we have $\nu_\ell(f)=\nu_\ell(K)$, and so by Proposition \ref{PropClassNumberRatio} the conjecture would be true.
\end{proof}

In the case where $g=3$, we can say more.

\begin{thm}\label{Thm:Main}
Let $f$ be a degree $6$ $q$-Weil polynomial that is ordinary, principally polarizable, cyclic, and maximal. Let $K$ be the splitting field of $f$ over $\Q$ and let $K^+$ be its maximal totally real subfield. Then 
\begin{equation*}\label{Eq:Main}
\nu_\infty(f)\prod_{\ell}\nu_\ell(f) = \frac{1}{\omega_K}\frac{h_{K}}{h_{K^+}}.
\end{equation*}
\end{thm}

\begin{proof}
Combine Proposition \ref{PropMatchSpecificCaseForEll} with Propositions \ref{PropMatchGeneralCaseForEll}, \ref{PropMatchGeneralCaseForP}, and \ref{PropClassNumberRatio} for $g=3$, and the theorem is proven.
\end{proof}

While Theorem \ref{Thm:Main} and Conjecture \ref{ConjMainThmGeneral} are interesting in their own right as an unexpected equality of a product of local densities of matrices to a ratio of class numbers, there is an interpretation of this quantity related to isogeny classes of abelian varieties via results in a preprint of Everett Howe \cite{howepreprint}.

We begin with some notation.  Given a Weil polynomial $f\in\Z[T]$, let $K$, $\pi_f$, $\bar\pi_f$, $K^+$, and $\O_f$ be defined as in section \ref{Sec:WeilPolys} and let $\O_{f^+}=\Z[\pi_f+\bar\pi_f]$.  Let $U$ be the unit group of $\O_f$ and let $U^+_{>0}$ be the group of totally positive units in $\O_{f^+}$.  Denote the narrow class number of $K^+$ by  $h^+_{K^+}$.  

\begin{thm}[Howe]\label{Thm:Howe}
Let $\mathcal{I}$ be an isogeny class of simple ordinary abelian varieties over $\Fq$ corresponding to an irreducible Weil polynomial $f\in\Z[T]$.  Using the notation above, suppose that $\O_f$ is the maximal order of $K$, and suppose that $K$ is ramified over $K^+$ at a finite prime.  Then the number of abelian varieties in $\mathcal{I}$ that have a principal polarization is equal to $\frac{h_K}{h^+_{K^+}}$, and each such variety has (up to isomorphism) exactly $[U^+_{>0}:N(U)]$ principal polarizations, where $N$ is the norm map from $U$ to $U^+_{>0}$. 
\end{thm}


We note that the conditions in Theorem \ref{Thm:Howe} are indeed met under our conditions; we assume that the abelian varieties are ordinary in \eqref{enord}, the polynomial $f$ is irreducible in \eqref{encyclic}, and $\O_f$ is the maximal order of $K$ in \eqref{enmax}.  By \cite[Lemma 10.2]{howe95}, $K$ is ramified over $K^+$ at a finite prime whenever $g$ is odd. 

\begin{cor}\label{CorToHowe}
Let everything be as in Theorem \ref{Thm:Howe}. If in addition the unit groups of $K$ and $K^+$ are equal, then the total number of principally polarized varieties lying in the isogeny class $\mathcal{I}$ is equal to $h_K/h_{K^+}$.
\end{cor}
\begin{proof}
If the unit groups of $K$ and $K^+$ are equal, then the ratio $\frac{h^+_{K^+}}{h_{K^+}}$ is equal to $[U^+_{>0}:N(U)]$ and the result follows from Theorem \ref{Thm:Howe}.
\end{proof}
\begin{rem}
Corollary \ref{CorToHowe} generalizes the classical result that the number of elliptic curves in a fixed isogeny class is given by the class number of an appropriate imaginary quadratic field.
\end{rem}

We assume in \eqref{encyclic} that $K/\Q$ is cyclic, which implies that Hasse's unit index is 1 (see the proof of Proposition~\ref{PropClassNumberRatio}).  Thus, under our conditions, the unit groups of $K$ and $K^+$ are equal, and so we can apply Corollary \ref{CorToHowe}.

Recall that $\A_g(\Fq; f)$ denotes the set of isomorphism classes of principally polarized abelian varieties of dimension $g$ over $\Fq$ with characteristic polynomial of Frobenius $f$, weighted inversely by the size of the automorphism group.  
Then we have the following corollary of Theorem~\ref{Thm:Main}.

\begin{cor}
\label{Cor:g3UseHowe}
Let $f$ be as in Theorem~\ref{Thm:Main}. Then $$\nu_{\infty}(f)\prod_{\ell}\nu_{\ell}(f) = \#\mathcal{A}_3(\F_q;f).$$
\end{cor}
\begin{proof}
From Corollary~\ref{CorToHowe}, the (unweighted) size of the isogeny class of principally polarized abelian threefolds with characteristic polynomial $f(T)$ is $\frac{h_{K}}{h_{K^+}}$.  The size of the automorphism group of any element of that isogeny class is $\omega_K$, so 
$$\frac{1}{\omega_K} \frac{h_K}{h_{K^+}} = \#\mathcal{A}_3(\F_q;f)$$
and the corollary is proven.
\end{proof}

If we also assume Conjecture \ref{ConjMainThmGeneral}, then we can state the following corollary, which generalizes the previous result.

\begin{cor}
Assume that Conjecture \ref{ConjMainThmGeneral} is true and let $f$ be as in that conjecture. Then $$\nu_{\infty}(f)\prod_{\ell}\nu_{\ell}(f) = \#\mathcal{A}_g(\F_q;f).$$
\end{cor}

Recent work of Marseglia gives an algorithm for computing isomorphism classes of abelian varieties in certain isogeny classes \cite{marseglia2018}.  One possible application of the algorithm, according to the author, is to provide computational evidence for the main formulas in \cite{achterwilliams2015}, \cite{achtergordon2017}, and the present work.

\subsection*{Acknowledgments}
We thank Everett Howe for sharing with us his work related to isogeny classes of principally polarized abelian varieties, and Jeff Achter for very helpful conversations.  We also thank the referee for useful suggestions and insightful questions.

\bibliographystyle{plain}
\bibliography{bibliography}

\end{document}